\documentclass[12pt,a4paper]{amsart}

\newtheorem{theo+}              {Theorem}           [section]
\newtheorem{prop+}  [theo+]     {Proposition}
\newtheorem{coro+}  [theo+]     {Corollary}
\newtheorem{lemm+}  [theo+]     {Lemma}
\newtheorem{exam+}  [theo+]     {Example}
\newtheorem{rema+}  [theo+]     {Remark}
\newtheorem{defi+}  [theo+]     {Definition}
\def \r{\mbox{${\mathbb R}$}}

\newenvironment{theorem}{\begin{theo+}}{\end{theo+}}
\newenvironment{proposition}{\begin{prop+}}{\end{prop+}}
\newenvironment{corollary}{\begin{coro+}}{\end{coro+}}
\newenvironment{lemma}{\begin{lemm+}}{\end{lemm+}}
\newenvironment{definition}{\begin{defi+}}{\end{defi+}}

\usepackage{amsthm}
\usepackage{color}
\theoremstyle{plain} \theoremstyle{remark}
\newtheorem{remark}{Remark}
\newtheorem{example}{Example}

\renewcommand{\Bbb}{\mathbb}
\newcommand{\rank}{\mbox{rank}}

\def\E{/\kern-1.0em \equiv }

\def\rank{rank}

\title{Some remarks on bi-$f$-harmonic maps and $f$-biharmonic maps}
\author{Yong Luo$^{*}$ and Ye-Lin Ou $^{**}$}
\thanks{$^{*}$  Yong Luo was supported by the NSF of China (No.11501421, No.11771339).\newline\indent
$^{**}$ Ye-Lin Ou was supported by a grant from the Simons Foundation ($\#427231$, Ye-Lin Ou).}
\address{School of Mathematics and Statistics, \newline\indent Wuhan University\newline\indent Wuhan 430072, \newline\indent China. \newline\indent E-mail: yongluo@whu.edu.cn \newline\indent
\newline\indent
Department of
Mathematics,\newline\indent Texas A $\&$ M University-Commerce,
\newline\indent Commerce, TX 75429,\newline\indent USA.\newline\indent
E-mail:yelin.ou@tamuc.edu }
\begin{document}

\title[Bi-$f$-harmonic maps $\&$ $f$-biharmonic maps]{Some remarks on bi-$f$-harmonic maps and $f$-biharmonic maps}

\subjclass{58E20, 53C43} \keywords{Biharmonic, $f$-biharmonic, and bi-$f$-harmonic maps, $f$-Laplacian, $f$-bi-Laplacian, and bi-$f$-Laplacian.}
\date{08/06/2018}
\maketitle

\section*{Abstract}
In this paper, we prove that  the class of  bi-$f$-harmonic maps and that of $f$-biharmonic maps from a conformal manifold of dimension $\ne 2$ are the same (Theorem \ref{MT1}).   We also give several results on nonexistence of proper bi-$f$-harmonic maps and  $f$-biharmonic maps  from complete Riemannian manifolds into nonpositively curved Riemannian manifolds. These include: any bi-$f$-harmonic map from a compact manifold into a non-positively curved manifold is $f$-harmonic (Theorem \ref{NPC}), and any $f$-biharmonic (respectively, bi-$f$-harmonic) map  with bounded $f$ and bounded $f$-bienrgy (respectively, bi-$f$-energy) from a complete Riemannian manifold into a manifold of strictly negative curvature has rank $\le 1$ everywhere (Theorems \ref{pro2} and \ref{pro4}).
\section{A relationship between $f$-biharmonic maps and bi-$f$-harmonic maps}

We assume all objects studied in this paper, including manifolds, tension fields, and maps,  are smooth unless they are stated otherwise.\\
{\em Harmonic maps} are critical points of the energy functional
\begin{equation}\notag
E(\phi)=\frac{1}{2}\int_\Omega |{\rm d}\phi|^2v_g
\end{equation}
for maps $\phi: (M,g)\longrightarrow (N,h)$ between Riemannian manifolds and for all  compact domain $\Omega \subseteq M$.  {\em Harmonic map equation } is simply the Euler Lagrange equation  of the energy functional which is given (see \cite{ES}) by
\begin{equation}\label{fhm}
\tau(\phi) \equiv {\rm Tr}_g\nabla\,d \phi=0,
\end{equation}
where $\tau(\phi)={\rm Tr}_g\nabla\,d \phi$ is  the tension field of
the map $\phi$.\\

{\em Biharmonic maps}  are critical points of the bienergy functional defined by
\begin{equation}\notag
E_{2}(\phi)=\frac{1}{2}\int_\Omega |\tau(\phi)|^2v_g,
\end{equation}
where $\Omega$ is a compact domain of $ M$ and $\tau(\phi)$ is the tension field of
the map $\phi$. {\em  Biharmonic map
equation} (see \cite{Ji}) is the
Euler-Lagrange equation of the bienergy functional, which can be written as
\begin{equation}\label{BTF}
\tau_{2}(\phi):={\rm
Tr}_{g}[(\nabla^{\phi}\nabla^{\phi}-\nabla^{\phi}_{\nabla^{M}})\tau(\phi)
-R^{N}({\rm d}\phi, \tau(\phi)){\rm d}\phi ]=0,
\end{equation}
where $R^{N}$ denotes the curvature operator of $(N, h)$ defined by
$$R^{N}(X,Y)Z=
[\nabla^{N}_{X},\nabla^{N}_{Y}]Z-\nabla^{N}_{[X,Y]}Z.$$ Clearly, any harmonic map is always a biharmonic map.\\

{\em $f$-Harmonic maps} (see \cite{Li}) are critical points of the $f$-energy functional defined by
\begin{equation}\notag
E_{f}(\phi)=\frac{1}{2}\int_\Omega f\,|{\rm d}\phi|^2v_g,
\end{equation}
where $\Omega$ is a compact domain of $ M$ and $f$ is a positive function on $M$.  {\em $f$-Harmonic map equation} can be written ( see, e.g., \cite{Co},
\cite{OND}) as
\begin{equation}\label{fhm}
\tau_f(\phi)\equiv f\tau(\phi)+{\rm d}\phi({\rm grad}\,f)=0,
\end{equation}
where $\tau(\phi)={\rm Tr}_g\nabla\,d \phi$ is the tension field of
$\phi$. It is easily seen that an $f$-harmonic map with $f=C$ is nothing but a harmonic map.\\

{\em $f$-Biharmonic maps} are critical points of the $f$-bienergy functional
\begin{equation}\notag
E_{f,2}(\phi)=\frac{1}{2}\int_\Omega f\,|\tau(\phi)|^2v_g,
\end{equation}
for maps $\phi: (M,g)\longrightarrow (N,h)$ between Riemannian manifolds and all compact domain $\Omega\subseteq M$. {\em $f$-Biharmonic map equation} can be written as (see \cite{Lu})
\begin{equation}\label{F2}
\tau_{f,2} (\phi)\equiv f\tau_2(\phi)+(\Delta f)\tau(\phi)+2\nabla^{\phi}_{{\rm grad}\, f}\tau(\phi)=0,
\end{equation}
where $\tau(\phi)$ and $\tau_2( \phi)$ are the tension and the bitension fields of
$\phi$ respectively. \\

{\em Bi-$f$-harmonic maps }  $\phi: (M,g)\longrightarrow (N,h)$ between Riemannian manifold which are critical points of the bi-$f$-energy functional
\begin{equation}\label{Bif}
E_{2,f}(\phi)=\frac{1}{2}\int_\Omega |\tau_f (\phi)|^2v_g,
\end{equation}
over all compact domain $\Omega$ of $ M$. The Euler-Lagrange
equation gives the bi-$f$-harmonic map equation ( \cite{OND})
\begin{equation}\label{2F}
\tau_{2,f} (\phi)\equiv -f J^{\phi}(\tau_f(\phi))+\nabla^{\phi}_{{\rm grad}\, f}\tau_f(\phi)=0,
\end{equation}
where $\tau_f(\phi)$ is the $f$-tension field of the map
$\phi$ and $J^{\phi}$ is the Jacobi operator of the map defined by $J^{\phi}(X)=-{\rm Tr}_g[\nabla^{\phi}\nabla^{\phi}X-\nabla^{\phi}_{\nabla^M}X-R^N(d \phi,\;X)d \phi$].

We would like to point out that in some literature, e.g., \cite{OND}, \cite{Ch}, the name ``$f$-biharmonic maps" was also used for the critical points of the bi-$f$-energy functional (\ref{Bif}).

From the above definitions, we can easily see the following inclusion relationships among the different types of biharmonic maps which  give two different paths of generalizations of harmonic maps.
\begin{eqnarray}\notag
\{ Harmonic\; maps\}\subset \{ Biharmonic \;maps\}\subset \{ f-Biharmonic \;maps\},\\\notag
\{Harmonic\; maps\}\subset \{ f-harmonic \;maps\} \subset \{Bi-f-harmonic \;maps\}.
\end{eqnarray}

For the obvious reason, we will call a bi-$f$-harmonic map which is not $f$-harmonic a {\em proper bi-$f$-harmonic map}, and an $f$-biharmonic map which is not harmonic a {\em proper $f$-biharmonic map}.\\

Concerning harmonicity and $f$-harmonicity, we have a classical result of Lichnerowicz \cite{Li} which states that  a map $\phi: (M^m, g)\longrightarrow (N^n, h)$ with $m \ne 2$ is an $f$-harmonic map if and only if it is a harmonic map with respect to a metric conformal to $g$. Concerning the relationship between biharmonic maps and $f$-biharmonic maps, it was proved in \cite{Ou2} that a map $\phi: (M^2, g)\longrightarrow (N^n, h)$ from a $2$-dimensional manifold is an $f$-biharmonic map if and only if it is a biharmonic map with respect to the metric ${\bar g}=f^{-1}g$ conformal to $g$.

Our first theorem shows that although $f$-biharmonic maps and bi-$f$-harmonic maps are quite different by their definitions, they belong to the same class in the sense that any bi-$f$-harmonic map $\phi: (M^m, g)\longrightarrow (N^n, h)$ with $m \ne 2$ and $f=\alpha$ is an $f$-biharmonic map for $f=\alpha^{m/(m-2)}$ with respect to some metric conformal to $g$, and conversely, any $f$-biharmonic map $\phi: (M^m, g)\longrightarrow (N^n, h)$ with $m \ne 2$ and $f=\alpha$ is a bi-$f$-harmonic map for $f=\alpha^{(m-2)/m}$ with respect to some metric conformal to $g$.
\begin{theorem}\label{MT1}
For  $m \ne 2$ and a map $\phi: (M^m, g)\longrightarrow (N^n, h)$ between Riemannian manifolds, we have
\begin{equation}\label{LK}
\tau_{2,f} (\phi,g)=f^{\frac{m}{m-2}}\,\tau_{f^{\frac{m}{m-2}},2}(\phi, \bar{g}),
\end{equation}
where $ \bar{g}=f^{\frac{2}{m-2}}g$. In particular, a map $\phi: (M^m, g)\longrightarrow (N^n, h)$ is a bi-$f$-harmonic map if and only if it is an $f^{\frac{m}{m-2}}$-biharmonic after the conformal change of the metric $ \bar{g}=f^{\frac{2}{m-2}}g$ in the domain manifold, conversely, a map $\phi: (M^m, g)\longrightarrow (N^n, h)$ is an $f$-biharmonic map if and only if it is a bi-$f^{\frac{m-2}{m}}$-harmonic after the conformal change of the metric $ \bar{g}=f^{-\frac{2}{m}}g$.
\end{theorem}
\begin{proof}
A straightforward computation gives the transformation of the tension fields under the conformal change
of a metric $\bar{g}=F^{-2}g$:
\begin{eqnarray}\notag
 \tau(\phi,{\bar g})&=& F^2\{\tau (\phi, g)-(m-2){\rm
d}{\phi}({\rm grad\,ln}F)\}.
\end{eqnarray}

When $m \ne 2$ and $F^{-2}=f^{\frac{2}{m-2}}$, we have
\begin{eqnarray}\notag
 \tau(\phi,{\bar g})&=& f^{\frac{-2}{m-2}}\tau (\phi, g)+f^{\frac{-m}{m-2}}{\rm
d}{\phi}({\rm grad\,}f)\\\notag &=& f^{\frac{-m}{m-2}}\left(f
\tau (\phi, g)+{\rm d}{\phi}({\rm grad\,}f)\right) =
f^{\frac{-m}{m-2}}\tau_f (\phi, g).
\end{eqnarray}
It follows that
\begin{eqnarray}\label{fTF}
\tau_f (\phi, g)=  f^{\frac{m}{m-2}} \tau(\phi,{\bar g}).
\end{eqnarray}

Now we use formula (3) in \cite{Ou1} to have the following formula of changes of Jacobi operators under the conformal change of metrics $ \bar{g}=f^{\frac{2}{m-2}}g$
\begin{equation}\notag
J_{g}(X) =f^{\frac{2}{m-2}}J_{\bar g}(X)+f^{-1}\nabla^{\phi}_{{\rm grad} f}(X).
\end{equation}
By using this and (\ref{2F}), we have
\begin{eqnarray}\notag
\tau_{2,f} ({\phi}, g)&= &-f J^{\phi}(\tau_f(\phi))+\nabla^{\phi}_{{\rm grad}\, f}\tau_f(\phi)\\\notag
&=&-f \left(f^{\frac{2}{m-2}}J_{\bar g}(\tau_f(\phi))+f^{-1}\nabla^{\phi}_{{\rm grad} f} \tau_f(\phi\right)+\nabla^{\phi}_{{\rm grad}\, f}\tau_f(\phi)\\\label{9}
&=&-f^{\frac{m}{m-2}}J_{\bar g}(\tau_f(\phi)).\\\notag
\end{eqnarray}
Substituting (\ref{fTF}) into (\ref{9}) yields
\begin{eqnarray}\label{GD10}
\tau_{2,f} (\phi,g)&=&-f^{\frac{m}{m-2}}J_{\bar g}\left(f^{\frac{m}{m-2}} \tau(\phi,{\bar g})\right).\\\notag
\end{eqnarray}
By a straightforward computation using (see (7) in \cite{Ou1})
\begin{equation}\notag
J(fX)=fJ(X)-(\Delta f)X-2\nabla^{\phi}_{{\rm grad} f}X,
\end{equation}
we can rewrite (\ref{GD10})  as
\begin{eqnarray}\notag
\tau_{2,f} (\phi)&=&-f^{\frac{m}{m-2}}J_{\bar g}\left(f^{\frac{m}{m-2}} \tau(\phi,{\bar g})\right) \\\notag
&=&f^{\frac{m}{m-2}}\left( -f^{\frac{m}{m-2}} J_{\bar g}(\tau(\phi,{\bar g})+ (\Delta_{\bar {g}} f^{\frac{m}{m-2}} )\tau(\phi,{\bar g})+2\nabla^{\phi}_{{\rm grad}_{\bar g} f^{\frac{m}{m-2}} }\tau(\phi,{\bar {g}})\right)\\\notag
&=&f^{\frac{m}{m-2}}\,\tau_{f^{\frac{m}{m-2}},2}(\phi, \bar{g}),
\end{eqnarray}
where the third equality was obtained by using the $f$-bi-tension field (\ref{F2}) and the fact that $-J_{\bar g}(\tau(\phi, {\bar g})=\tau_2(\phi, {\bar g})$. Thus, we obtain the relationship (\ref{LK}) from which the last statement of the theorem follows.
\end{proof}

Many examples of proper $f$-biharmonic maps were given in \cite{Ou2}, from which and Theorem \ref{MT1} we have
\begin{example}
It was proved in \cite{Ou2} that  for $m\ge 3$, the map $\phi: \r^m\setminus \{0\}\longrightarrow \r^m\setminus \{0\}$ with
$\phi(x)=\frac{x\;}{|x|^2}$ is a proper $f$-biharmonic map for $f(x)=|x|^4$.  Using Theorem \ref{MT1} we conclude that the map $\phi: (\r^m\setminus \{0\}, |x|^{-8/m}\delta_{ij}) \longrightarrow \r^m\setminus \{0\}$  with
$\phi(x)=\frac{x\;}{|x|^2}$ is a proper bi-$f$-harmonic map for $f(x)=|x|^{4(m-2)/m}$. In particular, when $m=4$, we have a bi-$f$-harmonic map $\phi: (\r^4\setminus \{0\}, |x|^{-2}\delta_{ij}) \longrightarrow \r^4\setminus \{0\}$  with
$\phi(x)=\frac{x\;}{|x|^2}$ for $f(x)=|x|^2$.
\end{example}

\begin{corollary}
A map $\phi: (M^m, g)\longrightarrow (N^n, h)$ with $m \ne 2$ is a bi-$f$-harmonic map  if and only if
\begin{eqnarray}
f^{\frac{m}{m-2}}\tau_2(\phi, {\bar g})+(\Delta_{\bar g} f^{\frac{m}{m-2}})\tau(\phi, {\bar g})+2\nabla^{\phi}_{{\rm grad}_{\bar g}f^{\frac{m}{m-2}}}\tau(\phi, {\bar g})=0,
\end{eqnarray}
where ${\bar g}=f^{\frac{2}{m-2}}g$, and $\tau(\phi, {\bar g})$ and $\tau_2(\phi, {\bar g})$ are the tension and the bi-tension fields of the map $\phi: (M^m, {\bar g})\longrightarrow (N^n, h)$ respectively.
\end{corollary}

\begin{proof}
This follows from Theorem \ref{MT1} and the $f^{\frac{m}{m-2}}$-biharmonic map equation (\ref{F2}).
\end{proof}

\begin{proposition}
A function $u: (M^m, g)\longrightarrow \r$ is bi-$f$-harmonic if and only if it is a solution of the following bi-$f$-Laplace equation
\begin{equation}
\Delta^2_f\,u=\Delta_f(\Delta_f\,u)=0,
\end{equation}
where $\Delta_f$ is the $f$-Laplace operator defined by
\begin{equation}\label{fDelta}
\Delta_f\,u=f\Delta\,u+g(\nabla f, \nabla u).
\end{equation}
\end{proposition}
\begin{proof}
For a real-valued function $u: (M^m, g)\longrightarrow \r$, one can easily check that the tension field is $\tau(u)=(\Delta u)\frac{\partial}{\partial t}$, and hence $\tau_f(u)=f (\Delta u)\frac{\partial}{\partial t}+du({\rm grad}f)=(\Delta_fu)\frac{\partial}{\partial t}$. It follows that the bi-$f$-tension field of $u$ is given by
\begin{eqnarray}\notag
\tau_{2, f}(u)&=& f J^{u}(\tau_f(u))-\nabla^{u}_{{\rm grad}\, f}\tau_f(u)\\\notag
&=&-f{\rm Tr}_g[\nabla^{u}\nabla^{u}\tau_f(u)-\nabla^{u}_{\nabla^M}\tau_f(u)-R^{N}(d u,\;\tau_f(u))d u]-\nabla^{u}_{{\rm grad}\, f}\tau_f(u)\\\notag
&=&-f{\rm Tr}_g[\nabla^{u}\nabla^{u}\tau_f(u)-\nabla^{u}_{\nabla^M}\tau_f(u)]-\nabla^{u}_{{\rm grad}\, f}\tau_f(u)\\\notag
&=&-\left(f{\sum_{i=1}^m[e_i(e_i (\Delta_fu))-\nabla^M_{e_i} e_i (\Delta_fu)]+\langle {\rm grad}\, f, \nabla (\Delta_fu)\rangle}\right)\frac{\partial}{\partial t}\\
&=&-[\Delta_f(\Delta_f u)]\frac{\partial}{\partial t}=-(\Delta_f^2 u)\frac{\partial}{\partial t},
\end{eqnarray}
from which the proposition follows.
\end{proof}

Following the practice of calling  $\Delta^2 \,u:=\Delta (\Delta \,u)$ the bi-Laplacian on $(M^m, g)$ we make the following definitions.
\begin{definition}
For a function $u: (M^m, g)\longrightarrow \r$ on a Riemannian manifold, we define the bi-$f$-Laplacian $\Delta_{2,f}$ acting on functions by
\begin{equation}
\Delta_{2,f} \,u:=\Delta_f^2 \,u=\Delta_f(\Delta_f \,u),
\end{equation}
and the $f$-bi-Laplacian by
\begin{equation}
\Delta_{f,2} \,u= \Delta(f\Delta u)=f\Delta^2 \,u+(\Delta\, f)\Delta \,u+2g(\nabla f, \nabla\Delta u).
\end{equation}
The solutions of the bi-$f$-Laplace equation $\Delta_f^2 \,u=0$ and that of the $f$-bi-Laplace equation $\Delta_{f,2} \,u=0$ are called bi-$f$-harmonic functions and $f$-biharmonic functions respectively.

\end{definition}

\begin{proposition}
For a real-valued function $u$ on a Riemannian manifold $(M^m, g)$ with $m \ne 2$, we have the following relationship between bi-$f$-Laplacian and $f^{\frac{m}{m-2}}$-bi-Laplacian operators
\begin{equation}
\Delta_{2,f}=\Delta_f^2 u= f^{\frac{m}{m-2}}{\bar \Delta}_{f^{\frac{m}{m-2}}, 2} u=f^{\frac{m}{m-2}} \Delta_{\bar g}\left(f^{\frac{m}{m-2}})\Delta_{\bar g} u\right),
\end{equation}
where ${\bar \Delta}_{f^{\frac{m}{m-2}}, 2}u:= f^{\frac{m}{m-2}}\Delta^2_{\bar g}u+(\Delta_{\bar g}f^{\frac{m}{m-2}})\Delta_{\bar g} u+2 \bar {g}({\rm grad}_{\bar {g}}f^{\frac{m}{m-2}}, {\rm grad}_{{\bar g}}\Delta^2_{\bar g} u)$ is the $f^{\frac{m}{m-2}}$-bi-Laplacian of the conformal metric ${\bar g}=f^{\frac{2}{m-2}}g$. In particular, a function $u: (M^m, g)\longrightarrow \r$ is bi-$f$-harmonic if and only if it is $f^{\frac{m}{m-2}}$-biharmonic with respect  the conformal metric ${\bar g}=f^{\frac{2}{m-2}}g$.
\end{proposition}
\begin{proof}
Using ${\bar \Delta} $ and ${\bar \nabla}$ to denote the Laplacian and the gradient operators on $(M^m, {\bar g}=f^{\frac{2}{m-2}}g)$, we can easily check that $\Delta_f u= f^{\frac{m}{m-2}}{\bar \Delta} u$. A straightforward computation yields
\begin{eqnarray}\notag
\Delta_f^2 \,u&=&\Delta_f(\Delta_f \,u)=f^{\frac{m}{m-2}}{\bar \Delta}(f^{\frac{m}{m-2}}{\bar \Delta}\,u)\\\notag
&=& f^{\frac{m}{m-2}}\left( f^{\frac{m}{m-2}}{\bar \Delta}^2\,u+({\bar \Delta}f^{\frac{m}{m-2}}){\bar \Delta}\,u+2{\bar g}({\bar \nabla} f^{\frac{m}{m-2}}, {\bar \nabla} {\bar \Delta}^2\,u )\right)\\
&=& f^{\frac{m}{m-2}}{\bar \Delta}_{ f^{\frac{m}{m-2}},2}\,u,
\end{eqnarray}
from which we obtain the proposition.
\end{proof}

In a recent paper \cite{Ch}, Chiang proved that any bi-$f$-harmonic map from a compact manifold without boundary into a non-positively curved manifold satisfying
\begin{equation}\label{KD}
\langle f\nabla^{\phi}_{e_i}\nabla^{\phi}_{e_i}\tau_f(\phi)-\nabla^{\phi}_{e_i}f\nabla^{\phi}_{e_i}\tau_f(\phi), \tau_f(\phi)\rangle\ge 0
\end{equation}
 is an $f$-harmonic map. Our next theorem gives an improvement of this result by dropping the condition (\ref{KD}), and hence gives a generalization of Jiang's result (Proposition 7 in \cite{Ji}) on biharmonic maps viewed as bi-$f$-harmonic maps with $f$ being a constant.
\begin{theorem}\label{NPC}
Any bi-$f$-harmonic map $\phi: (M^m, g)\longrightarrow (N^n, h)$ from a compact Riemannian manifold without boundary into a non-positively curved manifold  is an $f$-harmonic map.
\end{theorem}
\begin{proof}

A straightforward computation yields
 $$\frac{1}{2}\Delta|\tau_f(\phi)|^2=|\nabla^\phi\tau_f(\phi)|^2+\langle\Delta^\phi\tau_f(\phi), \tau_f(\phi)\rangle.$$
Using this and  Equation (\ref{2F}) we have
\begin{eqnarray*}
\frac{1}{2}\Delta|\tau_f(\phi)|^2&=&|\nabla^\phi\tau_f(\phi)|^2+\langle R^N(d\phi, \tau_f(\phi))d\phi, \tau_f(\phi)\rangle-f^{-1}\langle\nabla_{\nabla f}^\phi\tau_f(\phi), \tau_f(\phi)\rangle
\\&=&|\nabla^\phi\tau_f(\phi)|^2-R^N(d\phi, \tau_f(\phi),d\phi, \tau_f(\phi))-\frac{1}{2}\langle\nabla|\tau_f(\phi)|^2, \nabla\ln f\rangle.
\end{eqnarray*}
It follows that
\begin{eqnarray}\notag
&& \frac{1}{2}\Delta|\tau_f(\phi)|^2+\frac{1}{2}\langle\nabla|\tau_f(\phi)|^2, \nabla\ln f\rangle\\\label{GD50}
=&&|\nabla^\phi\tau_f(\phi)|^2-R^N(d\phi, \tau_f(\phi),d\phi, \tau_f(\phi))
\geq 0,
\end{eqnarray}
which implies that
\begin{eqnarray*}
&&\Delta_f \left( |\tau_f(\phi)|^2\right)\geq 0,
\end{eqnarray*}
where $\Delta_f$ is the $f$-Laplacian defined by (\ref{fDelta}). Applying the maximum principle we conclude that
$|\tau_f(\phi)|$ is constant and $\nabla^\phi\tau_f(\phi)=0$ on $M$.

To complete the proof, we define a vector field on $M$ by $Y=\langle f\nabla\phi, \tau_f(\phi)\rangle$. Then
\begin{eqnarray*}
divY=|\tau_f(\phi)|^2+\langle f\nabla\phi, \nabla^\phi\tau_f(\phi)\rangle=|\tau_f(\phi)|^2,
\end{eqnarray*}
which, together with Stokes's theorem, implies that $|\tau_f(\phi)|=0$, i.e. $\phi$ is an $f$-harmonic map. Thus, we obtain the theorem.
\end{proof}

From Theorem \ref{NPC} we have
\begin{corollary}\label{TE}
Any bi-$f$-harmonic map $\phi: (M^m, g)\longrightarrow \r^n$ from a compact manifold into a Euclidean space is a constant map.
\end{corollary}

From our Theorem \ref{NPC} and Theorem 1.1 in \cite{HLZ}  we see  that any bi-$f$-harmonic  or $f$-biharminic map from a closed Riemannian manifold into a Riemannian manifold of non-positive sectional curvature is an $f$-harmonic map or harmonic map respectively. Both these generalize Jiang's nonexistence result which states that biharmonic maps from  closed Riemannian manifolds into Riemannian manifolds of non-positive sectional curvature are harmonic maps (see Proposition 7 in \cite{Ji}). When $M$ is complete noncompact, nonexistence results of proper biharmonic maps into Riemannian manifolds of non-positive sectional curvature were first proved in \cite{BFO}, \cite{NUG} and later generalized in \cite{Ma} and \cite{Luo1}, \cite{Luo2}. In the rest of this paper, we will give some nonexistence results of proper bi-$f$-harmonic maps and $f$-biharmonic maps which generalize the corresponding results for biharmonic maps obtained in \cite{BFO}, \cite{NUG}, \cite{Ma}, \cite{Luo1}, and \cite{Luo2}.

\section{Some nonexistence theorems for proper bi-$f$-harmonic maps}

In this section, we give some nonexistence results of proper bi-$f$-harmonic maps from complete noncompact manifolds into a non-positively curved manifold.

 \begin{theorem}\label{pro}
 Let $\phi: (M, g)\longrightarrow(N, h)$ be a bi-$f$-harmonic map from a complete  Riemannian manifold  into a Riemannian manifold of non-positive sectional curvature. Then, we have

(i) If $(\int_M\, f |d\phi|^qdv_g)^\frac{1}{q}<+\infty$ and $\int_M|\tau_f(\phi)|^pfdv_g<\infty$ for some  $ q\in[1, \infty]$ and  $p\in (1, \infty)$,
then $\phi$ is an $f$-harmonic map.

(ii) If $Vol_f(M, g):=\int_Mfdv_g=\infty$ and $\int_M|\tau_f(\phi)|^pfdv_g<\infty$ for some  $p\in (1, \infty)$,
then $\phi$ is an  $f$-harmonic map.
 \end{theorem}
When the target manifold has strictly negative sectional curvature, we have
\begin{theorem}\label{pro2}
Let $\phi: (M, g)\longrightarrow (N, h)$ be a bi-$f$-harmonic map from a complete  Riemannian manifold  into a Riemannian manifold  of strictly negative sectional curvature with $\int_M|\tau_f(\phi)|^pfdv_g<\infty$ for some  $p\in (1, \infty)$. If there is some point $x\in M$ such that $\rank\phi(x)\geq2$, then $\phi$ is an $f$-harmonic map.
\end{theorem}

To prove these theorems, we will need the following lemma.
\begin{lemma}\label{main lem}
Let $\phi: (M, g)\longrightarrow (N, h)$ be a bi-$f$-harmonic map from a complete noncompact Riemannian manifold  into a Riemannian manifold  of non-positive sectional curvature. If $\int_M|\tau_f(\phi)|^pfdv_g<+\infty$ for some $p>1$, then $\nabla^\phi\tau_f(\phi)=0.$
\end{lemma}
\proof For a real number $\epsilon>0$, a straightforward computation shows that
\begin{eqnarray}\label{ine1}
&&\Delta(|\tau_f\phi)|^2+\epsilon)^\frac{1}{2}\nonumber
\\&=&(|\tau_f(\phi)|^2+\epsilon)^{-\frac{3}{2}}\left(\frac{1}{2}(|\tau_f(\phi)|^2+\epsilon)\Delta|\tau_f(\phi)|^2-\frac{1}{4}|\nabla|\tau_f(\phi)|^2|^2\right).
\end{eqnarray}
Since $\nabla|\tau_f(\phi)|^2=2h(\nabla^\phi\tau_f(\phi),\tau_f(\phi))$, we have
$$|\nabla|\tau_f(\phi)|^2|^2\leq 4(|\tau_f(\phi)|^2+\epsilon)|\nabla^\phi\tau_f(\phi)|^2,$$
from which we obtain
\begin{eqnarray}\label{GD29}
&&\frac{1}{2}(|\tau_f(\phi)|^2+\epsilon)\Delta|\tau_f(\phi)|^2-\frac{1}{4}|\nabla|\tau_f(\phi)|^2|^2 \nonumber
\\&\geq& \frac{1}{2}(|\tau_f(\phi)|^2+\epsilon)\left(\Delta|\tau_f(\phi)|^2-2|\nabla^\phi\tau_f(\phi)|^2\right).
\end{eqnarray}
Since $\phi$ is bi-$f$-harmonic, we use (\ref{2F}) to have
\begin{eqnarray}\label{ine3}
&&\frac{1}{2}\Delta|\tau_f(\phi)|^2\nonumber
\\&=&|\nabla^\phi\tau_f(\phi)|^2+\langle\Delta^\phi\tau_f(\phi), \tau_f(\phi)\rangle\nonumber
\\&=&|\nabla^\phi\tau_f(\phi)|^2-Tr_gR^N(\tau_f(\phi),d\phi,\tau_f(\phi),d\phi)-f^{-1}\langle\nabla_{\nabla f}^\phi\tau_f(\phi), \tau_f(\phi)\rangle\nonumber
\\&\geq& |\nabla^\phi\tau_f(\phi)|^2-\frac{1}{2}\langle\nabla|\tau_f(\phi)|^2,\nabla\ln f\rangle,
\end{eqnarray}
where the inequality was obtained by using the assumption that $R^N\leq 0$. Rewriting (\ref{ine3}) as
\begin{eqnarray}\label{GD30}
\Delta|\tau_f(\phi)|^2-2|\nabla^\phi\tau_f(\phi)|^2\ge -\langle\nabla|\tau_f(\phi)|^2,\nabla\ln f\rangle
\end{eqnarray}
and substituting (\ref{GD30}) into the right hand side of  (\ref{GD29}) we have
\begin{eqnarray}\label{ine2}
&&\frac{1}{2}(|\tau_f(\phi)|^2+\epsilon)\Delta|\tau_f(\phi)|^2-\frac{1}{4}|\nabla|\tau_f(\phi)|^2|^2\nonumber
\\&\geq&-\frac{1}{2}(|\tau_f(\phi)|^2+\epsilon)\langle\nabla|\tau_f(\phi)|^2,\nabla\ln f\rangle.
\end{eqnarray}
Using (\ref{ine1}) and (\ref{ine2})  we deduce that
$$\Delta(|\tau_f(\phi)|^2+\epsilon)^\frac{1}{2}\geq-\frac{1}{2}(|\tau_f(\phi)|^2+\epsilon)^{-\frac{1}{2}}\langle\nabla|\tau_f(\phi)|^2,\nabla\ln f\rangle.$$
Now taking the limit on both sides of the above inequality as $\epsilon\to 0$, we have that
\begin{equation}\label{GD31}
\Delta|\tau_f(\phi)|+\langle\nabla|\tau_f(\phi)|,\nabla\ln f\rangle\geq 0.
\end{equation}
This, by using the notation  $\Delta_{\alpha}:=\Delta-\langle\nabla \alpha,\nabla\cdot\rangle$, where $\alpha(x)$ is a  function on $M$ as in  \cite{WX}, can be written as
\begin{eqnarray}
\Delta_{-\ln f}|\tau_f(\phi)|\geq0,
\end{eqnarray}
which means that $|\tau_f(\phi)|$ is a positive $(-\ln f)$-subharmonic function on $M$. By Theorem 4.3 in \cite{WX}, we see that if $\int_M|\tau_f(\phi)|^pfdv_g<\infty$ for some $p>1$, then $|\tau_f(\phi)|$ is a constant on $M$. Moreover from (\ref{ine3}) we have $\nabla^\phi\tau_f(\phi)=0$. This completes the proof of the lemma. \endproof

\textbf{Proof of Theorem \ref{pro}.} From the above lemma we conclude that $|\tau_f(\phi)|=c$ is a constant. It follows that if $Vol_f(M)=\infty$, then we must have $c=0$, this proves (ii) of Theorem \ref{pro}. To prove (i) of Theorem \ref{pro}, we consider two cases. If $c=0$, then we are done. If $c\ne 0$, then the hypothesis implies that $Vol_f(M)<\infty$ and we will derive a contradiction as follows. Define a l-form on $M$ by
$$\omega(X):=\langle fd\phi(X),\tau_f(\phi)\rangle,~(X\in TM).$$
Then we have
\begin{eqnarray*}
\int_M|\omega|dv_g&=&\int_M(\sum_{i=1}^m|\omega(e_i)|^2)^\frac{1}{2}dv_g
\\&\leq&\int_Mf|\tau_f(\phi)||d\phi|dv_g
\\&\leq&c Vol_f(M)^{1-\frac{1}{q}}(\int_Mf|d\phi|^qdv_g)^\frac{1}{q}
\\&<&\infty ,
\end{eqnarray*}
where if $q=\infty$, we denote $\|d\phi\|_{L^\infty(M)}=(\int_Mf|d\phi|^qdv_g)^\frac{1}{q}$.

Now, we compute $-\delta\omega=\sum_{i=1}^m(\nabla_{e_i}\omega)(e_i)$ to have
\begin{eqnarray}\notag
-\delta\omega&=&\sum_{i=1}^m\nabla_{e_i}(\omega(e_i))-\omega(\nabla_{e_i}e_i)
\\\notag&=&\sum_{i=1}^m\{\langle\nabla^\phi_{e_i}(fd\phi(e_i)),\tau_f(\phi)\rangle
-\langle fd\phi(\nabla_{e_i}e_i),\tau_f(\phi)\rangle\}
\\\notag&=&\sum_{i=1}^m\langle f\nabla^\phi_{e_i}d\phi(e_i)-fd\phi(\nabla_{e_i}e_i)+\langle\nabla_{e_i} f,\nabla_{e_i}\phi\rangle,\tau_f(\phi)\rangle
\\\label{GD34}&=&|\tau_f(\phi)|^2,
\end{eqnarray}
where in obtaining the second equality we have used $\nabla^\phi\tau_f(\phi)=0$.  Now by Yau's generalized Gaffney's theorem (see Appendix) and (\ref{GD34}), we have
$$0=-\int_M\delta\omega dv_g=\int_M|\tau_f(\phi)|^2dv_g=c^2Vol(M),$$
which implies that $c=0$, a contradiction. Therefore we must have $c=0$, i.e. $\phi$ is an $f$-harmonic map. This completes the proof of theorem.
\endproof

\textbf{Proof of Theorem \ref{pro2}.} By Lemma \ref{main lem}, we know that $|\tau_f(\phi)|=c$,  a constant. We only need to prove that $c=0$. Assume that $c\neq 0$, we will derive a contradiction. It follows from (\ref{2F}) that at $x \in M$, we have
\begin{eqnarray*}
0&=&-\frac{1}{2}\Delta|\tau_f(\phi)|^2
\\&=&-\langle\Delta^\phi\tau_f(\phi), \tau_f(\phi)\rangle-|\nabla^\phi\tau_f(\phi)|^2
\\&=&\sum_{i=1}^m\langle R^N(\tau_f(\phi), d\phi(e_i))d\phi(e_i), \tau_f(\phi)\rangle+f^{-1}\langle\nabla_{\nabla f}^\phi\tau_f(\phi),\tau_f(\phi)\rangle-|\nabla^\phi\tau_f(\phi)|^2
\\&=&\sum_{i=1}^m\langle R^N(\tau_f(\phi), d\phi(e_i))d\phi(e_i), \tau_f(\phi)\rangle,
\end{eqnarray*}
where in obtaining the first and fourth equalities we have used Lemma \ref{main lem}. Since the sectional curvature of $N$ is strictly negative, we conclude that $d\phi(e_i)$ with $i=1, 2, \cdots, m$ are parallel to $\tau_f(\phi)$ at any  $x \in M$ and hence $\rank\phi(x)\leq1$.  This contradicts the assumption that $\rank\phi(x)\ge 2$ for some $x$. The contradiction shows that we must have $c=0$. Thus, we obtain the theorem.
\endproof

From Theorem \ref{pro}, we obtain the following corollary.
\begin{corollary}\label{main thm}
 Let $\phi: (M, g)\longrightarrow (N, h)$ be a bi-$f$-harmonic map with a bounded $f$ from a complete  Riemannian manifold into a Riemannian manifold  of non-positive sectional curvature. If \\
(i) $\int_Mf|d\phi|^2dv_g<+\infty$ and  $\int_M|\tau_f(\phi)|^2dv_g<\infty,$\\ or\\
(ii) $Vol_f(M, g):=\int_Mfdv_g=\infty$ and  $\int_M|\tau_f(\phi)|^2dv_g<\infty,$\\
then $\phi$ is $f$-harmonic.
\end{corollary}
Similarly, from Theorem \ref{pro2} we have the following corollary which shows that a proper bi-$f$-harmonic map with bounded $f$ and bounded bi-$f$-energy from a complete noncompact manifold into a negatively curved manifold must have $\rank \le 1$ at any point.
\begin{corollary}\label{main thm2}
Let $\phi: (M, g)\longrightarrow (N, h)$ be a bi-$f$-harmonic map with bounded $f$ from a complete  Riemannian manifold into a Riemannian manifold of strictly negative sectional curvature.  If $\int_M|\tau_f(\phi)|^2dv_g<\infty$ and there is a point $x\in M$ such that $\rank\phi(x)\geq 2$, then $\phi$ is an $f$-harmonic map.
\end{corollary}

\section{Some nonexistence theorems for proper $f$-biharmonic maps}
In this section we give some nonexistence theorems for proper $f$-biharmonic maps from a complete manifold into a non-positively curved manifold.  Such a study was started in \cite{Ou2} where it was proved that there exists no proper $f$-biharmonic map with constant $f$-bienergy density from a compact Riemannian manifold into a nonpositively curved manifold. Later in \cite{HLZ}, it was shown that the condition of having constant $f$-bienergy density can be dropped.  When the domain manifold is complete noncompact, the authors in \cite{HLZ} also proved the following result.
\begin{theorem}[HLZ]\label{He's thm}
Let $\phi:(M,g)\longrightarrow (N,h)$ be an $f$-biharmonic map from a complete  Riemannian manifold $(M,g)$ into a Riemannian manifold $(N,h)$ with non-positive curvature. If

(I) $ \int_M|d\phi|^2dv_g<\infty, \; \int_M|\tau(\phi)|^2dv_g<\infty, \;and \; \int_Mf^p|\tau(\phi)|^pdv_g<\infty $\\
for some $p\geq2$, or\\
(II) $Vol(M,g)=\infty$ and $\int_Mf^p|\tau(\phi)|^pdv_g<\infty$, \\
then $\phi$ is harmonic.
\end{theorem}

Our next theorem gives a generalization of Theorem \ref{He's thm}.
 \begin{theorem}\label{pro3}
An  $f$-biharmonic map  $\phi: (M, g)\longrightarrow (N, h)$  from a complete  Riemannian manifold  into a Riemannian manifold  of non-positive curvature is a harmonic map if either

(i) $(\int_M|d\phi|^qdv_g)^\frac{1}{q}<+\infty$ and $\int_Mf^p|\tau(\phi)|^pdv_g<\infty$ for some  $ q\in [1, \infty]$ and $p\in (1, \infty)$,
or
\\(ii) $Vol(M, g):=\int_Mdv_g=\infty$ and $\int_Mf^p|\tau(\phi)|^pdv_g<\infty$ for some  $ q\in [1, \infty]$ and $p\in (1, \infty)$.
 \end{theorem}

 \begin{remark}
 Clearly, our Theorem \ref{pro3} improves Theorem \ref{He's thm} since (I) and (II) in Theorem \ref{He's thm} (in which $p \ge 2$ and finite bienergy are required) implies (i) and (ii) in Theorem \ref{pro3} respectively, but not conversely.
 \end{remark}

When the target manifold has strictly negative sectional curvature, we have
\begin{theorem}\label{pro4}
Let $\phi: (M, g)\longrightarrow (N, h)$ be an f-biharmonic map from a complete  Riemannian manifold  into a Riemannian manifold  of strictly negative curvature with $\int_Mf^p|\tau(\phi)|^pdv_g<\infty$ for some $p\in (1,\infty)$. Assume that there is a point $x\in M$ such that $\rank\phi(x)\geq2$, then $\phi$ is a harmonic map.
\end{theorem}
The proofs of Theorems \ref{pro3} and \ref{pro4} are very similar to those of Theorems \ref{pro} and \ref{pro2}. The main ideas and outlines are given as follows.

\textbf{Proof of Theorem \ref{pro3}:} First, we rewrite the $f$-biharmonic map equation (\ref{F2}) as
\begin{eqnarray}\label{F2'}
\Delta^\phi(f\tau(\phi))-{\rm Trace}_{g} R^{N}({\rm d}\phi, f\tau(\phi)){\rm d}\phi=0,
\end{eqnarray}
 and use it to compute $\Delta(f^2|\tau(u)|^2+\epsilon)^\frac{1}{2}$ for a constant $\epsilon>0$.
Then, we use the result and an argument similar to that used in the proof of Lemma \ref{main lem} to have
$$\Delta(f^2|\tau(u)|^2+\epsilon)^\frac{1}{2}\geq0.$$
By taking limit on both sides of the above inequality as $\epsilon\to 0$, we have $\Delta(f|\tau(\phi)|)\geq0$. Now, by Yau's classical $L^p$ Liouville type theorem (see \cite{Yau}) and the assumption that $\int_Mf^p|\tau(\phi)|^pdv_g<\infty$, we conclude that $f|\tau(\phi)|=c$, a constant, which is used to deduce that $\nabla^\phi(f\tau(\phi))=0$.\\

It is easy to see that $f|\tau(\phi)|=c$ together with hypothesis (ii) implies that $c=0$, which mean $\tau(\phi)=0$ and hence $\phi$ is a harmonic map. To prove the theorem under hypothesis (i), we assume $c\neq0$ and define a field of  l-form $\omega$  on $M$ by  $$\omega(X):=\langle d\phi(X),f\tau(\phi)\rangle,~(X\in TM).$$
Then a straightforward computation shows that $-\delta\omega=f|\tau(\phi)|^2$ which, together with the assumption (i) of Theorem \ref{pro3}, implies that
$$\int_M|\omega|dv_g<\infty.$$
Using Yau's generalized Gaffney's theorem, we have $$0=\int_M\delta\omega dv_g=-\int_Mf|\tau(\phi)|^2dv_g=-c^2\int_Mf^{-1}dv_g,$$ which implies that $c=0$, a contradiction. The contradiction shows that we must have $c=0$,  and hence $\phi$ is a harmonic map. Thus, we obtain the theorem.

\textbf{Proof of Theorem \ref{pro4}:}  As in the first part of the proof of Theorem \ref{pro3}, we have $f|\tau(\phi)|=c$, a constant.  It is enough to prove that $c=0$. Assume that $c\neq 0$, we will derive a contradiction. Using  (\ref{F2'}) we have at $x\in M$:
\begin{eqnarray*}
0&=&-\frac{1}{2}\Delta (|f\tau(\phi)|)^2
\\&=&-\langle\Delta^\phi (f\tau(\phi)), f\tau(\phi)\rangle-|\nabla^\phi(f\tau(\phi)|^2
\\&=&\sum_{i=1}^m\langle R^N(f\tau(\phi), d\phi(e_i))d\phi(e_i), f\tau(\phi)\rangle-|\nabla^\phi f\tau(\phi)|^2
\\&=&\sum_{i=1}^m\langle R^N(f\tau(\phi), d\phi(e_i))d\phi(e_i), f\tau(\phi)\rangle.
\end{eqnarray*}
Since the sectional curvature of $N$ is strictly negative, we must have that $d\phi(e_i)$ with $i=1,2,\cdots, m$ are parallel to $f\tau(\phi)$ at any $x\in M$, so $\rank\phi(x)\leq1$. This contradicts the assumption that $\rank\phi(x)\geq2$ for some $x\in M$. The contradiction shows that $c=0$, and hence $\phi$ is a harmonic map. This completes the proof of Theorem \ref{pro4}.
\endproof

From Theorems \ref{pro3} and \ref{pro4} we have the following corollaries.
\begin{corollary}\label{main thm3}
Any $f$-biharmonic map $\phi:(M,g)\longrightarrow (N,h)$ with a bounded $f$  from a complete Riemannian manifold $(M,g)$ into a Riemannian manifold $(N,h)$ with non-positive curvature is harmonic if either
\\(i) $\int_M|d\phi|^2dv_g<\infty$ and $\int_Mf|\tau(\phi)|^2dv_g<\infty,$ \\or\\
(ii) $Vol(M,g)=\infty$ and $\int_Mf|\tau(\phi)|^2dv_g<\infty.$
\end{corollary}

\begin{corollary}\label{main thm4}
Any $f$-biharmonic map $\phi:(M,g)\longrightarrow (N,h)$ with a bounded $f$ and bounded $f$-bienergy from a complete Riemannian manifold into a Riemannian manifold  with strictly negative curvature is harmonic if there exists a point $x\in M$ such that $\rank(\phi)(x)\geq2$.
\end{corollary}

To  close this section, we give the following example which shows  that there does exist proper $f$-biharmonic map from a complete manifold into a non-positively curved manifold.
\begin{example}
It was proved in \cite{Ou1} (Page 141, Example 3) that the map $\phi: (\Bbb{R}^2, g=e^\frac{y}{R}(dx+dy^2))\longrightarrow\Bbb{R}^3, \phi(x,y)=(R\cos\frac{x}{R}, R\sin\frac{x}{R}, y)$ is a proper biharmonic conformal immersion of $\Bbb{R}^2$ into Euclidean space $\Bbb{R}^3$. Then by Theorem 2.3 of \cite{Ou2} we see that $\phi$ is a proper $f$-biharmonic map from a complete manifold $(\Bbb{R}^2, g=dx^2+dy^2)$ into Euclidean space $\Bbb{R}^3$ (a non-positively curved space) with $f=e^{-\frac{y}{R}}$.
\end{example}
Direct computations show that $d\phi=(-\sin\frac{x}{R}, \cos\frac{x}{R}, 0)dx+(0,0,1)dy$ and
$\tau(\phi)=-(\frac{1}{R}\cos\frac{x}{R}, \frac{1}{R}\sin\frac{x}{R}, 0)$.
Therefore, one can easily check that in this example we have
\begin{eqnarray}\notag
||d\phi||_{L^\infty(\Bbb{R}^2)}=\sqrt{2}<\infty, 
\\ \notag\ ||d\phi||_{L^q(\Bbb{R}^2)}=\sqrt{2}{\rm Vol} (\mathbb{R}^2)^\frac{1}{q}=\infty \ {\rm for\;}\;  q\in [1, \infty),
\\ \notag \;\;{\rm Vol} (\mathbb{R}^2)=\infty,\\\notag  \int_{\Bbb{R}^2}f^p|\tau(\phi)|^pdv_g=\infty(p>1), \;{\rm and}\; \int_{\Bbb{R}^2}f|\tau(\phi)|^2dv_g=\infty.
\end{eqnarray}
So, the example shows that  the hypothesis ``$\int_Mf^p|\tau(\phi)|^pdv_g<\infty$"  in Theorem \ref{pro3} and the hypothesis `` $f$ is bounded, $\int_M f|\tau(\phi)|^2dv_g<\infty$" in Corollary \ref{main thm3} cannot be dropped. Note also that $\rank(\phi)(x)=2$ for any $x\in \Bbb{R}^2$ and hence the hypothesis ``$\int_Mf^p|\tau(\phi)|^pdv_g<\infty$  and  the target manifold has strictly negative sectional curvature"  in Theorem \ref{pro4} and the hypothesis ``$f$ is bounded, $\int_M f|\tau(\phi)|^2dv_g<\infty$ and  the target manifold has strictly negative sectional curvature" in Corollary \ref{main thm4} can not de dropped.

\section{Further nonexistence results on proper $f$-biharmonic and bi-$f$-harmonic maps}
It would be interesting to know how and to what extent the nature of the function $f$ would affect the existence of proper $f$-biharmonic or bi-$f$-harmonic maps from  a complete manifold into a non-positively curved manifold. In particular, one would like to know whether the boundedness assumption on the function $f$ is essential in Corollaries \ref{main thm}, \ref{main thm2},  \ref{main thm3},  \ref{main thm4}. In this section, we will show that we can  weaken the boundedness assumption on $f$ by replacing it with
\begin{eqnarray}\label{ass1}
\sup_{B_r}f(x)=o(r^2), \ as\ r\to \infty,
\end{eqnarray}
or
\begin{eqnarray}\label{ass1'}
\sup_{B_r}f(x)\leq Cr^2F(r),
\end{eqnarray}
where $B_r$ is a geodesic ball of radius $r$ centered at some  point on $M$ and $F(r)$ is a nondecreasing function such that $\int_a^\infty\frac{1}{rF(r)}dr=\infty$ for some positive constant $a>0$. Here, we illustrate how to generalize Corollary \ref{main thm} by a weaker assumption. Recall that in the proof of Lemma \ref{main lem}, we have $\Delta_{-\ln f}|\tau_f(\phi)|\geq0$, then by Theorem 4.3 of \cite{WX}, we conclude that if
\begin{eqnarray}\label{ass2}
\overline{\lim}_{r\to\infty}\frac{1}{r^2F(r)}\int_{B_r}|\tau_f(\phi)|^2fdv_g<\infty,
 \end{eqnarray}
 then $|\tau_f(\phi)|$ is constant and furthermore $\nabla^\phi\tau_f(\phi)=0$. It is easy to see that under the assumption of  (\ref{ass1'}) and $\int_M|\tau_f(\phi)|^2dv_g<\infty$, we have
 \begin{eqnarray*}
\overline{\lim}_{r\to\infty}\frac{1}{r^2F(r)} \int_{B_r}f|\tau_f(\phi)|^2dv_g\leq\lim_{r\to\infty}\frac{\sup_{B_r}f(x)}{r^2F(r)}\int_M|\tau_f(\phi)|^2dv_g<\infty,
 \end{eqnarray*}
 which implies (\ref{ass2}). Therefore, $|\tau_f(\phi)|$ is a constant $c$ and furthermore $\nabla^\phi\tau_f(\phi)=0$. Thus, if $Vol_f(M)=\infty$, we must have $c=0$, i.e. $\phi$ is an $f$-harmonic map. This shows that under the hypothesis (ii) of Corollary \ref{main thm} and assumption (\ref{ass1'}), $\phi$ is an $f$-harmonic map. To see that Corollary \ref{main thm} holds under hypothesis (i) of Corollary \ref{main thm} and the assumption (\ref{ass1}),  we only need to prove that $c=0$ in this case. If otherwise, we see that $Vol(M)<\infty$ and we will derive a contradiction as follows. Define a l-form on $M$ by
$$\omega(X):=\langle fd\phi(X),\tau_f(\phi)\rangle,~(X\in TM).$$
Then we have
\begin{eqnarray*}
\underline{\lim}_{r\to\infty}\frac{1}{r}\int_{B_r}|\omega|dv_g&=&\underline{\lim}_{r\to\infty}\frac{1}{r}\int_{B_r}(\sum_{i=1}^m|\omega(e_i)|^2)^\frac{1}{2}dv_g
\\&\leq&\underline{\lim}_{r\to\infty}\frac{1}{r}\int_{B_r}f|\tau_f(\phi)||d\phi|dv_g
\\&\leq&c\underline{\lim}_{r\to\infty}\frac{1}{r}(\int_{B_r}fdv_g)^\frac{1}{2} (\int_Mf|d\phi|^2dv_g)^\frac{1}{2}
\\&\leq&\underline{\lim}_{r\to\infty}\frac{\sup_{B_r}\sqrt{f(x)}}{r}Vol(M)^\frac{1}{2}(\int_Mf|d\phi|^2dv_g)^\frac{1}{2}
\\&=&0.
\end{eqnarray*}
On the other hand, we compute $-\delta\omega=\sum_{i=1}^m(\nabla_{e_i}\omega)(e_i)$ to have
\begin{eqnarray*}
-\delta\omega&=&\sum_{i=1}^m\nabla_{e_i}(\omega(e_i))-\omega(\nabla_{e_i}e_i)
\\&=&\sum_{i=1}^m\{\langle\nabla^\phi_{e_i}(fd\phi(e_i)),\tau_f(\phi)\rangle
-\langle fd\phi(\nabla_{e_i}e_i),\tau_f(\phi)\rangle\}
\\&=&\sum_{i=1}^m\langle f\nabla^\phi_{e_i}d\phi(e_i)-fd\phi(\nabla_{e_i}e_i)+\langle\nabla_{e_i} f,\nabla_{e_i}\phi\rangle,\tau_f(\phi)\rangle
\\&=&|\tau_f(\phi)|^2,
\end{eqnarray*}
where in obtaining the second equality we have used $\nabla^\phi\tau_f(\phi)=0$.  Now by Yau's generalized Gaffney's theorem(see Appendix) and the above equality we have that
$$0=-\int_M\delta\omega dv_g=\int_M|\tau_f(\phi)|^2dv_g=c^2Vol(M),$$
which implies that $c=0$, a contradiction. Therefore we must have $c=0$, i.e. $\phi$ is an $f$-harmonic map.

Summarizing the above discussion, we have the following theorem which gives a generalization of Corollary \ref{main thm}.
\begin{theorem}\label{main thm'}
 Let $\phi: (M, g)\longrightarrow (N, h)$ be a bi-$f$-harmonic map from a complete Riemannian manifold $(M, g)$ into a Riemannian manifold $(N, h)$ of non-positive sectional curvature. If
 \\(i)
 \begin{eqnarray}
\lim_{r\to\infty}\frac{\sup_{B_r}f(x)}{r^2}=0,
 \end{eqnarray}
 where $B_r$ is a geodesic ball of radius $r$ around some point on $M$, and $$\int_Mf|d\phi|^2dv_g<+\infty \ and \ \int_M|\tau_f(\phi)|^2dv_g<\infty;$$
\\or
\\(ii)  \begin{eqnarray}
\sup_{B_r}f(x)\leq C F(r)r^2
 \end{eqnarray}
 for a nondecreasing function $F(r)$ such that $\int_a^\infty\frac{1}{rF(r)}dr=\infty$ for some positive constant $a>0$, and $$Vol_f(M, g):=\int_Mfdv_g=\infty\ and\ \int_M|\tau_f(\phi)|^2dv_g<\infty,$$
then $\phi$ is $f$-harmonic.
\end{theorem}
Similar arguments apply to obtain  the following theorems, which give generalizations of the corresponding results in Corollaries \ref{main thm2}, \ref{main thm3} and \ref{main thm4}.
\begin{theorem}\label{main thm2'}
Let $\phi: (M, g)\longrightarrow (N, h)$ be a bi-$f$-harmonic map from a complete Riemannian manifold $(M, g)$ into a Riemannian manifold $(N, h)$ of strictly negative sectional curvature.  Assume that
 \begin{eqnarray}
\sup_{B_r}f(x)\leq C F(r)r^2,
 \end{eqnarray}
 where $B_r$ is a geodesic ball of radius $r$ around some point on $M$ and $F(r)$ is a nondecreasing function such that $\int_a^\infty\frac{1}{rF(r)}dr=\infty$ for some positive constant $a>0$.  Then if $$\int_M|\tau_f(\phi)|^2dv_g<\infty$$ and there is some point $x\in M$ such that $\rank\phi(x)\geq2$, $\phi$ is an $f$-harmonic map.
\end{theorem}

\begin{theorem}\label{main thm3'}
Let $\phi:(M,g)\longrightarrow(N,h)$ be an $f$-biharmonic map from a complete Riemannian manifold $(M,g)$ into a Riemannian manifold $(N,h)$ of non-positive curvature. If
 \\(i)
 \begin{eqnarray}
\lim_{r\to\infty}\frac{\sup_{B_r}f(x)}{r^2}=0,
 \end{eqnarray}
 where $B_r$ is a geodesic ball of radius $r$ around some point on $M$,
 $$\int_M|d\phi|^2dv_g<\infty, \ and \ \int_Mf|\tau(\phi)|^2dv_g<\infty;$$
\\or
\\(ii) \begin{eqnarray}
\sup_{B_r}f(x)\leq C F(r)r^2
 \end{eqnarray}
 for  a nondecreasing function $F(r)$  such that $\int_a^\infty\frac{1}{rF(r)}dr=\infty$ for some positive constant $a>0$,
 $$Vol(M,g)=\infty, \ and  \ \int_Mf|\tau(\phi)|^2dv_g<\infty,$$
then $\phi$ is a harmonic map.
\end{theorem}

\begin{theorem}\label{main thm4'}
Let $\phi:(M,g)\longrightarrow (N,h)$ be an $f$-biharmonic map from a complete Riemannian manifold $(M,g)$ into a Riemannian manifold $(N,h)$ of strictly negative curvature. Assume that
 \begin{eqnarray}
\sup_{B_r}f(x)\leq C F(r)r^2,
 \end{eqnarray}
 where $B_r$ is a geodesic ball of radius $r$ around some point on $M$ and $F(r)$ is a nondecreasing function such that $\int_a^\infty\frac{1}{rF(r)}dr=\infty$ for some positive constant $a>0$.
Then if $\int_Mf|\tau(\phi)|^2dv_g<\infty$ and there exists a point $x\in M$ such that $\rank(\phi)(x)\geq2$, $\phi$ is a harmonic map.
\end{theorem}
 \quad

\section{Appendix}
\begin{theorem}[Yau's generalized Gaffney's theorem]
Let $(M, g)$ be a complete Riemannian manifold. If $\omega$ is a $C^1$ 1-form such
that $\underline{\lim}_{r\to\infty}\int_{B_r}|\omega|dv_g=0$, or equivalently, a $C^1$ vector field $X$ defined by
$\omega(Y) = \langle X, Y \rangle, (\forall Y \in TM)$ satisfying $\underline{\lim}_{r\to\infty}\int_{B_r}|X|dv_g=0$, where $B_r$ is a geodesic ball of radius $r$ around some point on $M$, then $$\int_M\delta\omega dv_g=\int_Mdiv Xdv_g=0.$$
\end{theorem}
\proof See Appedix in \cite{Yau}.
\vskip1cm
\textbf{Acknowledgements:}
A part of the work was done when Yong Luo
was a visiting scholar at Tsinghua University. He would like to express his gratitude to
Professors Yuxiang Li and Hui Ma for their invitation and to Tsinghua University for the hospitality.


\begin{thebibliography}{99}
\bibitem[BFO]{BFO} P. Baird, A. Fardoun and S. Ouakkas, {\em Liouville-type theorems for biharmonic maps between Riemannian manifolds},
Adv. Calc. Var. 3 (2010), 49--68.
\bibitem[Ch]{Ch} Yuan-Jen Chiang, {\em f-biharmonic Maps between Riemannian Manifolds}, Geom. Integrability $\&$ Quantization,
Proceedings of the Fourteenth International Conference on Geometry, Integrability and Quantization, Iva�lo M. Mladenov, Andrei Ludu and Akira Yoshioka, eds. (Sofia: Avangard Prima, 2013), 74--86.
\bibitem[Co]{Co} N. Course, {\em f-harmonic maps}, Thesis, University of Warwick, Coventry, CV4 7AL, UK, 2004.
\bibitem[ES]{ES} J. Eells and J. H. Sampson, {\em Harmonic mappings of Riemannian manifolds}, Amer. J. Math. 86(1964), 109--160.
\bibitem[HLZ]{HLZ} G. He, J. Li, and P. Zhao, {\em Some results of $f$-biharmonic maps into a Riemannian manifold of non-positive sectional curvature}, Bull. Korean Math. Soc. 54 (2017), No. 6, pp. 2091-2106.
\bibitem[Ji]{Ji} G. Y. Jiang, {\em 2-harmonic maps and their first and second variational formulas,} Chinese Ann. Math. Ser. A7(1986), 389--402.
Translated into English by H. Urakawa in Note Mat. 28(2009), Suppl. 1, 209--232.
376-383.
\bibitem[Li]{Li} A. Lichnerowicz, {\em Applications harmoniques et vari$\acute{\rm e}$t$\acute{\rm e}$s k$\ddot{\rm a}$hleriennes(French)}, Rend. Sem. Mat. Fis. Milano 39(1969), 186--195.
\bibitem[Lu]{Lu} Wei-Jun Lu, {\em On f-biharmonic maps and bi-f-harmonic maps between Riemannian manifolds}, Science China Math. 58(7)(2015), 1483-1498.
\bibitem[Luo1]{Luo1} Y. Luo, {\em Liouville type theorems on complete manifolds and non-existence of bi-harmonic maps,} J. Geom. Anal. 25(2015), 2436--2449.
\bibitem[Luo2]{Luo2} Y. Luo, {\em Remarks on the nonexistence of biharmonic maps,} Arch. Math. (Basel)107(2016), no.2, 191--200.
\bibitem[Ma]{Ma} S. Maeta, {\em Biharmonic maps from a complete Riemannian manifold into a non-positively curved manifold,} Ann. Glob. Anal. Geom. 46(2014), 75--85.
\bibitem[NUG]{NUG} N. Nakauchi, H. Urakawa, and S. Gudmundsson, {\em Biharmonic maps into a Riemannian manifold of non-positive curvature}, Geom Dedicata 169(2014), 263-272.
\bibitem[Ou1]{Ou1} Y. -L. Ou, {\em On conformal biharmonic immersions}, Ann. Global Analysis and Geometry, 36(2) (2009), 133--142.
\bibitem[Ou2]{Ou2} Y. -L. Ou, {\em On $f$-biharmonic maps and $f$-biharmonic submanifolds},   Pacific J. of Math. 271-2 (2014), 461-477.
\bibitem[OND]{OND} S. Ouakkas, R. Nasri, and M. Djaa, {\em On the $f$-harmonic and $f$-biharmonic maps}, JP J. Geom. Topol. 10 (1)(2010), 11-27.
\bibitem[WX]{WX} G. F. Wang and D. L. Xu,  Harmonic maps from smooth metric measure spaces, {\em Internat. J. Math.} {\bf23}(2012), no.9, 1250095, 21 pp.
\bibitem[Yau]{Yau} S. T. Yau, {\em Some function-theoretic properties of complete Riemannian manifold and their applications to geometry},
Indiana Unvi. Math. J. 25(1976), no.7, 659--670.

\end{thebibliography}
\end{document}